\documentclass[paper=a4,fontsize=12pt,abstract]{scrartcl}
\pdfoutput=1 
\usepackage[ngerman,english]{babel,varioref}
\usepackage[utf8]{inputenc} 
\usepackage[T1]{fontenc} 
\usepackage{lmodern,microtype} 
\usepackage{csquotes} 

\usepackage{amssymb}
\usepackage{mathtools}
\usepackage{amsthm}
\usepackage{pgf,tikz} 
\usepackage{dsfont} 

\usepackage{setspace} 

\usepackage{booktabs} 
\usepackage{paralist}

\deffootnote{1.8em}{0em}{\textsuperscript{\thefootnotemark}\ }

\makeatletter
\g@addto@macro\bfseries{\boldmath}
\makeatother 

\newtheoremstyle{mythm}
{\topsep}   
{\topsep}   
{\itshape}  
{0pt}       
{\bfseries} 
{.}         
{5pt plus 1pt minus 1pt} 
{}          

\newtheoremstyle{mydefi}
{\topsep}   
{3ex}   
{\normalfont}  
{0pt}       
{\bfseries} 
{.}         
{5pt plus 1pt minus 1pt} 
{}          

\makeatletter
\newenvironment{myproof}[1][\proofname]{\par
	\pushQED{\qed}%
	\normalfont \topsep6\p@\@plus6\p@\relax
	\trivlist
	\item[\hskip\labelsep
	\itshape
	#1\@addpunct{.}]\ignorespaces
}{%
	\popQED\endtrivlist\@endpefalse\bigskip
}
\makeatother

\theoremstyle{mythm}
\newtheorem{satz}{Satz}[section]
\newtheorem{thm}[satz]{Theorem}
\newtheorem{prop}[satz]{Proposition}
\newtheorem{lem}[satz]{Lemma}

\newtheorem{cor}[satz]{Corollary}

\newtheorem{rem}[satz]{Remark}

\theoremstyle{mydefi}
\newtheorem{defi}[satz]{Definition}

\newtheorem{ex}[satz]{Example}

\DeclareMathOperator{\Aut}{Aut}
\DeclareMathOperator{\SL}{SL}
\DeclareMathOperator{\GL}{GL}
\DeclareMathOperator{\PSL}{PSL}
\DeclareMathOperator{\PGL}{PGL}
\DeclareMathOperator{\PGammaL}{P\Gamma L}
\DeclareMathOperator{\ord}{ord}
\DeclareMathOperator{\N}{N}

\newcommand{\inv}[1]{#1^{-1}}
\newcommand{\mat}[4]{\left( \begin{smallmatrix}#1&#2\\#3&#4\end{smallmatrix} \right)}
\newcommand{\zv}[2]{\mbox{\small\ensuremath{(#1,#2)}}}
\newcommand{\sv}[2]{\left(\begin{smallmatrix}#1\\#2\end{smallmatrix}\right)}

\usepackage[backend=biber,hyperref]{biblatex}
\usepackage[colorlinks=true,linkcolor=.,citecolor=.,urlcolor=blue,breaklinks=true,%
	pdftitle={Parallelisms and Translations of (Affine) SL(2,q)-Unitals},%
	pdfauthor={Verena Möhler}%
	]{hyperref}

\title{Parallelisms and Translations of (Affine) \texorpdfstring{$\SL(2,q)$-Unitals}{SL(2,q)-Unitals}}
\author{Verena Möhler}
\date{December 17, 2020}

\begin{document}

\maketitle

\begin{abstract}
	\noindent Unitals can be obtained as closures of affine unitals via parallelisms. The isomorphism type of the closure depends on the chosen parallelism, which need not be unique. For affine $\SL(2,q)$-unitals, we introduce a class of parallelisms for odd order and one for square order.	
	
	\noindent Translations are automorphisms of unitals, fixing each block through a given center.	For each of the known parallelisms of affine $\SL(2,q)$-unitals, we compute all possible translations with centers on the block at infinity.\bigskip
	
	\smallskip
	
\end{abstract}

Most of the results in the present paper have been obtained in the author's Ph.\,D.\ thesis \cite{diss}, where detailed arguments can be found for some statements that
we leave to the reader here.

\section{Preliminaries}

Projective planes arise from affine planes via a unique parallelism. The strategy of building the affine part of a geometry first and then completing it by using a parallelism on the blocks can successfully be applied to other incidence structures than affine and projective planes. We apply this approach to a special construction of unitals. A~crucial difference to affine planes is, that for a given affine unital, there may be various parallelisms, leading to non-isomorphic closures.\bigskip

A \textbf{unital of order $n$} is a $2$-$(n^3+1,n+1,1)$ design, i.\,e.\ an incidence structure with $n^3+1$ points, $n+1$ points on each block and unique joining blocks for any two points. We also consider \emph{affine unitals}, which arise from unitals by removing one block (and all the points on it) and can be completed to unitals via a parallelism on the short blocks. We give an axiomatic description:

\goodbreak
\begin{defi}
	Let $n\in\mathbb N$, $n\geq2$. An incidence structure $\mathbb U=(\mathcal P,\mathcal B,\emph I)$ is called an \textbf{affine unital of order \boldmath $n$} if:
	\begin{itemize}
		\item[(AU1)] There are $n^3-n$ points.
		\item[(AU2)] Each block is incident with either $n$ or $n+1$ points. The blocks incident with $n$~points will be called \textbf{short blocks} and the blocks incident with $n+1$ points will be called \textbf{long blocks}.
		\item[(AU3)] Each point is incident with $n^2$ blocks.
		\item[(AU4)] For any two points there is exactly one block incident with both of them.
		\item[(AU5)] There exists a \textbf{parallelism} on the short blocks, meaning a partition of the set of all short blocks into $n+1$ parallel classes of size $n^2-1$ such that the blocks of each parallel class are pairwise non-intersecting.
	\end{itemize}	
\end{defi}

The existence of a parallelism as in (AU5) must explicitly be required (see \cite[Example 3.10]{diss}). An affine unital $\mathbb U$ of order $n$ with parallelism $\pi$ can be completed to a unital $\mathbb U^\pi$ of order $n$ as follows: For each parallel class, add a new point that is incident with each short block of that class. Then add a single new block $[\infty]^\pi$, incident with the $n+1$ new points (see \cite[Proposition 3.9]{diss}). We call $\mathbb U^\pi$ the \textbf{$\pi$-closure} of $\mathbb U$.\bigskip

Note that though we must require the existence of a parallelism in the definition of an affine unital, this parallelism need not be unique. It is therefore not convenient to require that isomorphisms of affine unitals respect certain parallelisms and we will only ask them to be isomorphisms of incidence structures. We call two parallelisms $\pi$ and $\pi'$ of an affine unital $\mathbb U$ \textbf{equivalent} if there is an automorphism of $\mathbb U$ which maps $\pi$ to $\pi'$.\bigskip

Given an affine unital $\mathbb U$ with parallelisms $\pi$ and $\pi'$, the closures $\mathbb U^\pi$ and $\mathbb U^{\pi'}$ are isomorphic with $[\infty]^\pi \mapsto [\infty]^{\pi'}$ exactly if $\pi$ and $\pi'$ are equivalent (see \cite[Proposition 3.12]{diss}).\bigskip

\section{Parallelisms of Affine \texorpdfstring{$\SL(2,q)$-Unitals}{SL(2,q)-Unitals}}

From now on let $p$ be a prime and $q\coloneqq p^e$ a $p$-power. We are interested in a special kind of affine unitals, namely affine $\SL(2,q)$-unitals. The construction of those affine unitals is due to Grundhöfer, Stroppel and Van Maldeghem \cite{slu}. They consider translations of unitals, i.\,e.\ automorphisms fixing each block through a given point (the so-called center). Of special interest are unitals of order $q$ where two points are centers of translation groups of order $q$. In the classical (Hermitian) unital of order $q$, any two such translation groups generate a group isomorphic to $\SL(2,q)$; see \cite[Main Theorem]{moufang} for further possibilities. The construction of (affine) $\SL(2,q)$-unitals is motivated by this action of $\SL(2,q)$ on the classical unital.\bigskip

Let $S\leq \SL(2,q)$ be a subgroup of order $q+1$ and let $T\leq \SL(2,q)$ be a Sylow $p$-subgroup. Recall that $T$ has order $q$ (and thus trivial intersection with $S$), that any two conjugates $T^h\coloneqq\inv hTh$, $h\in \SL(2,q)$, have trivial intersection unless they coincide and that there are $q+1$ conjugates of $T$.

Consider a collection $\mathcal D$ of subsets of $\SL(2,q)$ such that each set $D\in \mathcal D$ contains $\mathds{1}\coloneqq \mat 1001$, that $\# D=q+1$ for each $D\in \mathcal D$, and the following properties hold:

\begin{itemize}
	\item[(Q)] For each $D\in \mathcal D$, the map
	\[(D\times D)\smallsetminus \{(x,x) \mid x\in D\} \to \SL(2,q)\text{,}\quad (x,y)\mapsto x\inv y\text{,}\]
	is injective, i.\,e.\ the set $D^*\coloneqq\{x\inv y \mid x,y\in D,\ x\neq y\}$
	contains $q(q+1)$ elements.
	\item[(P)] The system consisting of $S\smallsetminus\{\mathds{1}\}$, all conjugates of $T\smallsetminus\{\mathds{1}\}$ and all sets $D^*$ with $D\in \mathcal D$ forms a partition of $\SL(2,q)\smallsetminus\{\mathds{1}\}$.
\end{itemize}

Set
\begin{align*}
\mathcal P \coloneqq &\ \SL(2,q)\text{,}\\
\mathcal B \coloneqq &\ \{Sg \mid g\in \SL(2,q)\} \cup \{T^h g \mid h,g\in \SL(2,q)\} \cup \{Dg \mid D\in\mathcal D, g\in\SL(2,q)\}
\end{align*}
and let the incidence relation $I\subseteq \mathcal P\times\mathcal B$ be containment.\bigskip

Then we call the incidence structure $\mathbb U_{S,\mathcal D}\coloneqq (\mathcal P,\mathcal B,I)$ an \textbf{\boldmath affine $\SL(2,q)$-unital}. Each affine $\SL(2,q)$-unital is indeed an affine unital of order $q$, see \cite[Prop. 3.15]{diss}. If $\pi$ is a parallelism on the short blocks of an affine $\SL(2,q)$-unital $\mathbb U_{S,\mathcal D}$, we call the $\pi$-closure~$\mathbb U_{S,\mathcal D}^\pi$ an \textbf{$\SL(2,q)$-($\pi$-)unital}.\bigskip

For the investigation of parallelisms of affine unitals, we only need to know how the short blocks look like. In any affine $\SL(2,q)$-unital, the set of short blocks is the set of all right cosets of the Sylow $p$-subgroups. A parallelism as in (AU5) means a partition of the set of short blocks into $q+1$ sets of $q^2-1$ pairwise non-intersecting blocks.

\begin{defi} Let $\mathfrak P$ denote the set of all Sylow $p$-subgroups of $\SL(2,q)$.
	\begin{enumerate}[(a)]
		\item We denote by $\mathfrak S\coloneqq \{Tg \mid T\in\mathfrak P, g\in\SL(2,q)\}$ the set of short blocks of any affine $\SL(2,q)$-unital.
		\item A \textbf{parallelism on $\mathfrak S$} is a partition of $\mathfrak S$ into $q+1$ sets of $q^2-1$ pairwise non-intersecting cosets.
	\end{enumerate}
\end{defi}

Note that each right coset $Tg$ is a left coset $gT^g$ of a conjugate of $T$. For each prime power $q$, there are hence two obvious parallelisms on $\mathfrak S$, namely partitioning the set of short blocks into the sets of \emph{right} cosets or into the sets of \emph{left} cosets of the Sylow $p$-subgroups. We name those two parallelisms ``flat'' and ``natural'', respectively, and denote them by the corresponding musical signs
\[ \flat \coloneqq \{\{Tg\mid g\in\SL(2,q)\}\mid T\in\mathfrak P\}\quad \text{and}\quad \natural \coloneqq \{\{gT\mid g\in\SL(2,q)\}\mid T\in\mathfrak P\}\text{.} \]

\begin{rem}
	Inspired by the construction of $\SL(2,q)$-unitals from affine $\SL(2,q)$-unitals via different parallelisms, Nagy and Mez{\H o}fi created a method of constructing new incidence structures from old ones by removing a block and attaching it again in a different way, see \emph{\cite{paramod}}. They call this method \emph{paramodification} and used it to 
	construct plenty new unitals of orders $3$ and $4$, respectively. Their new unitals also comprise the Leonids unitals \emph{(}see \emph{Section \ref{leonids})}, since these are obtained from the classical unital of order $4$ by paramodification.
\end{rem}

On any affine $\SL(2,q)$-unital $\mathbb U_{S,\mathcal D}$, the group $\SL(2,q)$ acts as a group of automorphisms by multiplication from the right. We denote this permutation group by $R\coloneqq \{\rho_h\mid h\in\SL(2,q)\}$, where $\rho_h$ acts by right multiplication with $h\in\SL(2,q)$. Any automorphism of $\SL(2,q)$ induces a permutation on the point set of $\mathbb U_{S,\mathcal D}$; we call the corresponding permutation group $\mathfrak A \cong \Aut(\SL(2,q))$. For each $a\in\GL(2,q)$, conjugation by $a$ induces a permutation $\gamma_a\in\mathfrak A$. Although not every $\alpha\in\mathfrak A$ acts as automorphism on $\mathbb U_{S,\mathcal D}$, we know that for $q\geq 3$, the full automorphism group of $\mathbb U_{S,\mathcal D}$ is a subgroup of $\mathfrak A\ltimes R$, more precisely a subgroup of $\mathfrak A_S\ltimes R$ (see \cite[Theorem 3.3]{autos}). We are therefore also interested in the stabilizers of our parallelisms under the action of $\mathfrak A\ltimes R$.\bigskip

Regarding the parallelisms $\flat$ and $\natural$, we see that both are stabilized by $\mathfrak A\ltimes R$. Note that right multiplication with $h\in\SL(2,q)$ fixes each parallel class of $\flat$ while it acts on the parallel classes of $\natural$ via conjugation on the Sylow $p$-subgroups $T\in\mathfrak P$ (since $gTh=ghT^h$).

\goodbreak

\begin{ex}\leavevmode
	\begin{enumerate}[(a)]
		\item For each prime power $q$ we may choose $S=C$ to be cyclic and $\mathcal H$ a set of blocks through $\mathds{1}$ such that $\mathbb U_{C,\mathcal H}$ is isomorphic to the affine part of the classical unital and the closure $\mathbb U_{C,\mathcal H}^\natural$ is isomorphic to the classical unital. We call $\mathbb U_{C,\mathcal H}$ the \textbf{classical affine $\SL(2,q)$-unital} and $\mathbb U_{C,\mathcal H}^\natural$ the \textbf{classical $\SL(2,q)$-unital}. See \cite[Example 3.1]{slu} or \cite[Section 3.2.2]{diss} for details.
		\item Several non-classical affine $\SL(2,q)$-unitals are known, namely one of order $4$, described in \cite{slu}, and three of order $8$, described in \cite[Section 3]{threeaff}.
	\end{enumerate}
\end{ex}

\subsection{A Class of Parallelisms for Odd Order} \label{rpa}

For each odd prime power $q$, there is at least one class of parallelisms on $\mathfrak S$ apart from $\flat$ and $\natural$. Let $q$ be odd throughout this section.\bigskip

Let $T$ be a fixed Sylow $p$-subgroup of $\SL(2,q)$, namely
\[ T\coloneqq\{\mat1x01 \mid x\in\mathbb F_q\}\text{.} \]
The normalizer of $T$ in $\GL(2,q)$ is the set of upper triangular matrices. Let $\mathbb F_q^\square$ denote the set of all squares in $\mathbb F_q$, let $\mathbb F_q^{\times,\square}$ denote the set of all squares in $\mathbb F_q^\times$ and $\mathbb F_q^{\times,\not\square}$ the set of all non-squares in $\mathbb F_q^\times$. Note that $\#\mathbb F_q^{\times,\square}=\#\mathbb F_q^{\times,\not\square}=\frac 12 (q-1)$ since $q$ is odd.
Let
\[\Omega\coloneqq \{g\in\SL(2,q)\mid \mbox{\small\ensuremath{(0,1)}}\cdot g\cdot \left(\begin{smallmatrix}1\\0\end{smallmatrix}\right) \in \mathbb F_q^\square\}\quad \text{and} \quad \Omega^\mathsf{c}\coloneqq \SL(2,q)\smallsetminus \Omega\text{.}\]

Note that $\Omega$ is a union of $T$-cosets (left as well as right cosets) since
\[ \zv01 \cdot Tg \cdot \sv10 = \{\zv01 \cdot g\cdot \sv10\} = \zv01 \cdot gT \cdot \sv10\text{.}\]

Let further
\begin{align*}
A &\coloneqq \{Tg \mid g\in \Omega\} \cup \{gT \mid g\in \Omega^\mathsf{c}\}\text{,}\\
A' &\coloneqq \{Tg \mid g\in \Omega^\mathsf{c}\} \cup \{gT \mid g\in \Omega\}
\end{align*}
and\footnote{%
	Note that in \cite{diss}, the parallelism $\pi_{\mathsf{odd}}$ is called $\pi^\square$ and $\pi_{\mathsf{odd}}'$ is called ${}^{\square}\pi$.
}
\[ \pi_{\mathsf{odd}} \coloneqq \{ A^h\mid h\in \SL(2,q)\}\text{,}\quad  \pi_{\mathsf{odd}}' \coloneqq \{ A'^h\mid h\in \SL(2,q)\}\text{.} \]

\begin{thm}\label{rpalpa}
	For odd $q$, the sets $\pi_{\mathsf{odd}}$ and $\pi_{\mathsf{odd}}'$ are parallelisms on $\mathfrak S$. With $v\in\mathbb F_q^{\times,\not\square}$, conjugation by $\mat100v$ maps $\pi_{\mathsf{odd}}$ to $\pi_{\mathsf{odd}}'$.
\end{thm}

\begin{myproof}
	Since $\Omega$ is a union of $T$-cosets, the cosets in $A$ are pairwise non-intersecting and we get
	\[ \# A = \frac{\#\Omega}{\#T} + \frac{\#\Omega^\mathsf{c}}{\#T} = \frac{\#\SL(2,q)}{\#T} = q^2-1 \text{.}\]
	Hence, $A$ is a set of $q^2-1$ pairwise non-intersecting short blocks.\bigskip
	
	Let $h\in\N\coloneqq\N_{\SL(2,q)}(T)$. Then $h$ is upper triangular, say $h\coloneqq\mat a*0{\inv a}$ with $a\in\mathbb F_q^\times$. For any $g\in\SL(2,q)$, we have
	\[ \zv01 \cdot g^h\cdot \sv10 = \zv0a \cdot g\cdot \sv a0 = a^2\zv01\cdot g\cdot \sv10 \text{.}\]
	Since $a^2\in\mathbb F_q^{\times,\square}$, we get $g^h\in\Omega$ exactly if $g\in\Omega$. Hence, conjugation by~$h$ stabilizes $A$.\bigskip
	
	Now let $h\in\SL(2,q)\smallsetminus\N$ and $Tg\in A$. Then $(Tg)^h= T^hg^h$ is not a right coset of $T$. Compute $T^hg^h=g^h(T^h)^{g^h} = g^hT^{gh}$ and consider $gh\in\N$. Then $h=\inv gn$ for an $n\in\N$ and we have $g^h = g^n$. But since $g^n\in\Omega$ exactly if $g\in\Omega$ (as shown above), the coset $(Tg)^h$ is not contained in $A$. A similar consideration shows that $(gT)^h$ is not contained in $A$ if $gT\in A$ and $h\notin \N$.\bigskip
	
	Let $g,h\in\SL(2,q)$ and assume $A^h\cap A^g\neq\varnothing$. Then $A^{h\inv g}\cap A\neq\varnothing$ and we get $h\inv g\in\N$, $A^{h\inv g}=A$ and hence $A^h=A^g$. Thus, $\pi_{\mathsf{odd}} = \{A^h\mid h\in\SL(2,q)\}$ is indeed a partition of $\mathfrak S$ into $\#\SL(2,q)/\#\N = q+1$ sets of $q^2-1$ pairwise non-intersecting cosets, hence a parallelism on $\mathfrak S$.\bigskip
	
	Now let $v\in\mathbb F_q^{\times,\not\square}$ and $f\coloneqq\mat 100v$. Then $f\in\N_{\GL(2,q)}(T)$ and $g^f\in\N$ exactly if $g\in\N$. Further, for any $g\in\SL(2,q)\smallsetminus\N$, we have
	\[ \zv01 \cdot g^f\cdot \sv10 = \zv0{\inv v} \cdot g\cdot \sv 10 = \inv v\zv01\cdot g\cdot \sv10 \text{.} \]
	Hence, $g^f\in\Omega\smallsetminus\N$ exactly if $g\in\Omega^\mathsf{c}$ and we get $A^f=A'$ (note that $\N\subseteq\Omega$). For each $A^h\in \pi_{\mathsf{odd}}$, we get
	\[ (A^h)^f = A^{hf} = (A^f)^{(h^f)} = A'^{(h^f)} \in \pi_{\mathsf{odd}}'\text{.} \]
	Hence, we obtain $\pi_{\mathsf{odd}}'$ from $\pi_{\mathsf{odd}}$ via conjugation by $f$, and $\pi_{\mathsf{odd}}'$ is a parallelism on $\mathfrak S$ since $\pi_{\mathsf{odd}}$ is a parallelism on $\mathfrak S$.	
\end{myproof}

As mentioned above and according to \cite[Theorem 3.3]{autos}, the full automorphism group of any affine $\SL(2,q)$-unital of order $q\geq 3$ is a subgroup of $\mathfrak A\ltimes R$. We compute the stabilizer of the parallelism $\pi_{\mathsf{odd}}$ in $\mathfrak A\ltimes R\cong \Aut(\SL(2,q))\ltimes \SL(2,q)$ (note that $q$ is odd and hence $q\geq 3$). Recall that
$\Aut(\SL(2,q)) = \PGammaL(2,q) \coloneqq \PGL(2,q)\rtimes \Aut(\mathbb F_q)$, where automorphisms of $\mathbb F_q$ act entrywise on matrices. We denote by $\varphi$ the Frobenius automorphism $x\mapsto x^p$ on~$\mathbb F_q$. 

\begin{thm}\label{stabrpa}
	Let $c\coloneqq \mat 100{-1}$. The stabilizer of $\pi_{\mathsf{odd}}$ in the group $\mathfrak A\ltimes R$ equals
	\begin{enumerate}[\em (a)]
		\item $\mathrm{P\Sigma L}(2,q)\times \langle \rho_{-\mathds{1}}\rangle$ if $q\equiv 1\mod 4$ and
		\item $\mathrm{P\Sigma L}(2,q)\rtimes \langle \gamma_c \cdot \rho_{-\mathds{1}}\rangle$ if $q\equiv 3\mod 4$,
	\end{enumerate}
	where $\mathrm{P\Sigma L}(2,q)\coloneqq \PSL(2,q)\rtimes \Aut(\mathbb F_q)$.
\end{thm}

\begin{myproof}
	Note first that the Frobenius automorphism $\varphi$ stabilizes $\Omega$ and $T$ and hence $A$ and $A'$. Thus, for each $A^h\in\pi_{\mathsf{odd}}$, we have
	\[ A^h\cdot \varphi = (A\cdot\varphi)^{h\cdot\varphi} = A^{h\cdot\varphi} \in \pi_{\mathsf{odd}} \]
	and $\langle\varphi\rangle=\Aut(\mathbb F_q)$ stabilizes $\pi_{\mathsf{odd}}$ and equally $\pi_{\mathsf{odd}}'$.\bigskip
	
	The action of $\PSL(2,q)$ obviously stabilizes $\pi_{\mathsf{odd}}$ and $\pi_{\mathsf{odd}}'$ by construction. Since the index of $\PSL(2,q)$ in $\PGL(2,q)$ equals $2$, the orbit of $\pi_{\mathsf{odd}}$ under the action of $\PGL(2,q)$ has length $1$ or $2$. From Theorem \ref{rpalpa}, we know that conjugation by $\mat 100v$ ($v\in\mathbb F_q^{\times,\not\square}$) maps $\pi_{\mathsf{odd}}$ to $\pi_{\mathsf{odd}}'$, and hence the stabilizer of $\pi_{\mathsf{odd}}$ in the action of $\PGL(2,q)$ equals $\PSL(2,q)$ and conjugation by any element in $\PGL(2,q)\smallsetminus \PSL(2,q)$ interchanges $\pi_{\mathsf{odd}}$ and $\pi_{\mathsf{odd}}'$.\bigskip
	
	Since we know now that $\PGammaL(2,q)$ stabilizes $\{\pi_{\mathsf{odd}},\pi_{\mathsf{odd}}'\}$, we need to find those elements in $R$ which also stabilize the set of these two parallelisms. Let $\rho_g\in R$ and assume $\pi_{\mathsf{odd}}\cdot\rho_g\in\{\pi_{\mathsf{odd}},\pi_{\mathsf{odd}}'\}$. Since for each $A^h\in\pi_{\mathsf{odd}}$ the automorphism $\rho_g$ maps the set of right cosets of $T^h$ in $A^h$ on a set of right cosets of $T^h$, it must then also map the set of left cosets of $T^h$ in $A^h$ on a set of left cosets of $T^h$. Hence, $g$ is contained in the normalizer of every Sylow $p$-subgroup of $\SL(2,q)$ and thus $g\in\{\pm \mathds{1}\}$. We see immediately that
	\[ \pi_{\mathsf{odd}}\cdot\rho_{-\mathds{1}} = \begin{cases} \pi_{\mathsf{odd}} &\text{ if } -1\in\mathbb F_q^{\times,\square}\text{,}\\ \pi_{\mathsf{odd}}' &\text{ if } -1\in\mathbb F_q^{\times,\not\square}\text{.} \end{cases}\]
	If $q\equiv 1\mod 4$, then $-1\in\mathbb F_q^{\times,\square}$ and $\rho_{-\mathds{1}}$ stabilizes $\pi_{\mathsf{odd}}$. Since $\rho_{-\mathds{1}}$ does not fix~$\mathds{1}$ -- while every automorphism in $\mathfrak A$ does -- and since $\rho_{-\mathds{1}}$ commutes with every automorphism in~$\mathfrak A$, statement (a) follows.
	
	If $q\equiv 3\mod 4$, then both $\rho_{-\mathds{1}}$ and conjugation by $c$ interchange $\pi_{\mathsf{odd}}$ and $\pi_{\mathsf{odd}}'$ and hence the product stabilizes $\pi_{\mathsf{odd}}$. Again, $\gamma_c\cdot\rho_{-\mathds{1}}$ does not fix $\mathds{1}$ -- while every automorphism in~$\mathfrak A$ does -- and hence the product $\mathrm{P\Sigma L}(2,q)\cdot \langle \gamma_c \cdot \rho_{-\mathds{1}}\rangle$ is semidirect, since $\rho_{-\mathds{1}}$ commutes with every automorphism in~$\mathfrak A$ and $\gamma_c$ normalizes $\mathrm{P\Sigma L}(2,q)\leq \mathfrak A$.
\end{myproof}

\begin{rem}
	Knowing the full automorphism group of $\mathbb U_{C,\mathcal H}$, we use \emph{Theorem \ref{stabrpa}} to compute that the stabilizer of $[\infty]^{\pi_{\mathsf{odd}}}$ in the full automorphism group of the closure $\mathbb U_{C,\mathcal H}^{\pi_{\mathsf{odd}}}$ has order $2e(q+1)$ \emph{(}see \emph{\cite[Corollary 5.3]{diss}}\emph{)}.
\end{rem}

\subsection{A Class of Parallelisms for Square Order}
For each square order, there also is at least one class of parallelisms on $\mathfrak S$ apart from $\flat$ and $\natural$. Consider the quadratic field extension $\mathbb F_{q^2}/\mathbb F_q$ and the unique involutory field automorphism
\[ \overline{\cdot}\colon \mathbb F_{q^2}\to \mathbb F_{q^2}\text{, } x\mapsto \overline x\coloneqq x^q\text{,} \]
with fixed field $\mathbb F_q$. We let this automorphism act entrywise on matrices over $\mathbb F_{q^2}$.
Let again $T\coloneqq\{\mat 1x01\mid x\in\mathbb F_{q^2}\}$ be a fixed Sylow $p$-subgroup of $\SL(2,q^2)$. Let further
\[\Omega\coloneqq \{g\in\SL(2,q^2)\mid \mbox{\small\ensuremath{(0,1)}}\cdot g\cdot \left(\begin{smallmatrix}1\\0\end{smallmatrix}\right) \in \mathbb F_q\}\quad \text{and} \quad \Omega^\mathsf{c}\coloneqq \SL(2,q^2)\smallsetminus \Omega\text{.}\]

As in Section \ref{rpa}, $\Omega$ is a union of $T$-cosets. Consider the action

\[ \vartheta \colon \mathbb F_{q^2}^{2\times 2} \times \GL(2,q^2) \to \mathbb F_{q^2}^{2\times 2},\ (x,h)\mapsto x\cdot\vartheta_h \coloneqq \inv hx\overline h \text{,} \]

and note that $\vartheta_h$ equals conjugation with $h$ if $h\in\GL(2,q)$.

\begin{lem}\label{omegainv}
	Let $h\in\N_{\SL(2,q^2)}(T)$. Then $\Omega\cdot\vartheta_h=\Omega$ and $\Omega^\mathsf{c}\cdot\vartheta_h=\Omega^\mathsf{c}$.
\end{lem}
\begin{myproof}
	Since $h$ normalizes $T$, we have $h$ upper triangular, say $h\coloneqq \mat a*0{\inv a}$ with $a\in\mathbb F_{q^2}^\times$. For any $g\in\SL(2,q^2)$, we have
	\[ \mbox{\small\ensuremath{(0,1)}}\cdot (g\cdot\vartheta_h)\cdot \left(\begin{smallmatrix}1\\0\end{smallmatrix}\right) = \mbox{\small\ensuremath{(0,a)}}\cdot g\cdot \left(\begin{smallmatrix}\overline a\\0\end{smallmatrix}\right) = a\overline a\cdot\mbox{\small\ensuremath{(0,1)}}\cdot g\cdot \left(\begin{smallmatrix}1\\0\end{smallmatrix}\right) \text{.} \]
	Since $a\overline a\in\mathbb F_q^\times$, we get $g\cdot \vartheta_h\in\Omega$ exactly if $g\in\Omega$.
\end{myproof}

\begin{thm}
Let
\[ A\coloneqq\{Tg\mid g\in\Omega\}\cup\{gT\mid g\in\Omega^\mathsf{c}\}\quad \text{and}
\quad \pi_{\mathsf{sq}}\coloneqq\{ A\cdot \vartheta_h\mid h\in\SL(2,q^2)\} \text{.} \]
	Then $\pi_{\mathsf{sq}}$ is a parallelism on $\mathfrak S$.
\end{thm}

\begin{proof}
	Note first that the cosets in $A$ are pairwise non-intersecting and that
	\[ \# A = \frac{\#\Omega}{\#T} + \frac{\#\Omega^\mathsf{c}}{\#T} = \frac{\#\SL(2,q^2)}{\#T} = q^4-1 \text{.}\]
	Hence, $A$ is indeed a set of $(q^2)^2-1$ non-intersecting short blocks.\bigskip

	Let $h\in\N\coloneqq\N_{\SL(2,q^2)}(T)$ and $Tg\in A$. Then
	\[ (Tg)\cdot\vartheta_h = \inv hTg\overline h = \inv h Th\inv hg\overline h = T^h(g\cdot\vartheta_h) = T(g\cdot\vartheta_h) \in A\text{,}\]
	since $\vartheta_h$ leaves $\Omega$ invariant (Lemma \ref{omegainv}). Analogously, $(gT)\cdot\vartheta_h\in A$ if $gT\in A$ and hence $A\cdot\vartheta_h=A$ for $h\in\N$.\bigskip
	
	Now let $h\in\SL(2,q^2)\smallsetminus\N$ and $Tg\in A$. Then $(Tg)\cdot\vartheta_h=T^h(g\cdot\vartheta_h)$ is not a right coset of $T$. Assume $(Tg)\cdot\vartheta_h=(g\cdot\vartheta_h)T^{g\overline h}$ is a left coset of $T$, i.\,e.\ $n\coloneqq g\overline h\in\N$. Hence,
	\[g\cdot\vartheta_h=\inv h g\overline h= \inv{\overline n}\overline gn = \overline g\cdot\vartheta_{\overline n} \text{,}\]
	which is in $\Omega$ exactly if $\overline g \in\Omega$. But since obviously $\overline\Omega = \Omega$ and $g\in\Omega$ since $Tg\in A$, we have $g\cdot\vartheta_h\in\Omega$ and hence $(Tg)\cdot\vartheta_h\notin A$. A similar consideration shows that $(gT)\cdot\vartheta_h\notin A$ if $gT\in A$. Thus, $(A\cdot\vartheta_h)\cap A=\varnothing$ if $h\in\SL(2,q^2)\smallsetminus\N$.\bigskip
	
	Finally, let $g,h\in\SL(2,q^2)$ and assume $(A\cdot\vartheta_g)\cap(A\cdot\vartheta_h)\neq\varnothing$. Then $(A\cdot\vartheta_{g\inv h})\cap A\neq\varnothing$ and hence $g\inv h\in\N$, $A\cdot\vartheta_{g\inv h}=A$ and $A\cdot\vartheta_g=A\cdot\vartheta_h$. Thus, $\pi_{\mathsf{sq}}=\{A\cdot\vartheta_h\mid h\in\SL(2,q^2)\}$ is indeed a partition of $\mathfrak S$ into $\#\SL(2,q^2)/\#\N=q^2+1$ sets of $(q^2)^2-1$ pairwise non-intersecting cosets, hence a parallelism on $\mathfrak S$.
	
\end{proof}

\goodbreak

\begin{rem}\label{pioddsq} \leavevmode
	\begin{enumerate}[\em (a)]
		\item For each parallel class $X$ of $\pi_{\mathsf{sq}}$, there exists a unique Sylow $p$-subgroup $P\in\mathfrak P$ such that $X$ consists of $q^3-1$ right cosets of $P$ (of which $q^2-1$ are left cosets of $\overline P$) and of $q^4-q^3$ left cosets of $\overline P$, none of which is a right coset of $P$.
		\item Applying inversion on $\SL(2,q^2)$ to the parallelism $\pi_{\mathsf{sq}}$ yields another parallelism $\inv{\pi_{\mathsf{sq}}}$. Since no element of $\mathfrak A\ltimes R$ maps $\pi_{\mathsf{sq}}$ to $\inv{\pi_{\mathsf{sq}}}$, the two parallelisms are not equivalent in any affine $\SL(2,q^2)$-unital.
		\item  If $q^2$ is odd, neither $\pi_{\mathsf{sq}}$ nor $\inv{\pi_{\mathsf{sq}}}$ are equivalent to $\pi_{\mathsf{odd}}$.
	\end{enumerate}
\end{rem}

As for $\pi_{\mathsf{odd}}$, we compute the stabilizer of $\pi_{\mathsf{sq}}$ in $\mathfrak A\ltimes R$. We abbreviate $\Gamma\coloneqq \mathrm{Stab}_{\mathfrak A\ltimes R}(\pi_{\mathsf{sq}})$.

\goodbreak
\begin{lem} \label{fqqq}
	Let $x\in\mathbb F_{q^2}^\times$ and $c\in\mathbb F_q^\times$ such that $x\overline x= c^2$. Then $x\in\mathbb F_{q^2}^{\times,\square}$.
\end{lem}

\begin{proof}
	If $q$ is even, then every element in $\mathbb F_{q^2}$ is a square. Let $q$ be odd and let $z$ be a generator of $\mathbb F_{q^2}^\times$. Note that $\mathbb F_q^\times=\{z^{l(q+1)}\mid l\in\mathbb Z\}$ and let $k,l\in\mathbb Z$ such that $x=z^k$ and $c=z^{l(q+1)}$. Then
	\[ z^{k(q+1)} = x\overline x = c^2 = z^{2l(q+1)} \text{,}\]
	and we get $k-2l\equiv 0\mod q^2-1$. Since $q^2-1$ is even, we get that $k$ is even and hence $x$ is a square.
\end{proof}

\goodbreak

\begin{lem}\label{lemstabqpa} \leavevmode
	\begin{enumerate}[\em (a)]
		\item $\Aut(\mathbb F_{q^2})\leq \Gamma$.
		\item Let $r\in\SL(2,q^2)$. We have $\rho_r \in\Gamma$ exactly if $r\in\{\pm\mathds{1}\}$.
		\item Let $d\in\mathbb F_{q^2}^\times$ and $a\coloneqq\mat d001$. Then $d\in\mathbb F_{q^2}^{\times,\square}$ exactly if there exists $r\in\SL(2,q^2)$ with $\gamma_a\rho_r\in\Gamma$.
		\item Let $a\in\GL(2,q^2)$. Then $\det(a)\in\mathbb F_{q^2}^{\times,\square}$ exactly if there exists $r\in\SL(2,q^2)$ with $\gamma_a\rho_r\in\Gamma$.
		\item Let $a\in\GL(2,q^2)$, let $r,s\in\SL(2,q^2)$ and assume $\gamma_a\rho_r,\gamma_a\rho_s\in\Gamma$. Then $s\in\{\pm r\}$.
	\end{enumerate}
\end{lem}

\begin{proof}
	\begin{enumerate}[(a)]
		\item As in the proof of Theorem \ref{stabrpa}, we have that the Frobenius automorphism $\varphi\in\Aut(\mathbb F_{q^2})$ stabilizes $\Omega$ and $T$ and hence $A$. Further, $\varphi$ and $\overline\cdot$ commute and we thus get $\pi_{\mathsf{sq}}\cdot\varphi=\pi_{\mathsf{sq}}$ and hence $\Aut(\mathbb F_{q^2})\leq \Gamma$.
		
		\item Note first that $\Omega\cdot\rho_{-\mathds{1}}=\Omega$ and hence $A\cdot\rho_{-\mathds{1}}=A$. Since for any $h$, the permutation~$\rho_{-\mathds{1}}$ commutes with $\vartheta_h$, we get $\pi_{\mathsf{sq}}\cdot\rho_{-\mathds{1}}=\pi_{\mathsf{sq}}$.
		
		Let $r\in\SL(2,q^2)$ and $P\in\mathfrak P$. Recall Remark \ref{pioddsq}
		(a). The permutation $\rho_r$ maps each set of right cosets of $P$ to a set of right cosets of $P$ and each set of left cosets of $\overline P$ to a set of left cosets of $\overline P^r$. Hence, if $\rho_r\in\Gamma$, then $r\in\N_{\SL(2,q^2)}(P)$ for any $P\in\mathfrak P$ and thus $r\in\{\pm\mathds{1}\}$.
		
		\item Let $d=c^2\in\mathbb F_{q^2}^{\times,\square}$ and set $a'\coloneqq \inv c\cdot a \in\SL(2,q^2)$. Set $r\coloneqq \inv{a'}\overline{a'}$. Then $\gamma_a\rho_r=\gamma_{a'}\rho_r= \vartheta_{a'} \in \Gamma$.
		
		Conversely, let $d\in\mathbb F_{q^2}^\times$ and assume that there exists $r\in\SL(2,q^2)$ such that $\gamma_a\rho_r\in\Gamma$. As above, $\gamma_a\rho_r$ maps each set of right cosets of $P$ to a set of right cosets of $P^a$ and each set of left cosets of $\overline P$ to a set of left cosets of $\overline P^{ar}$. Hence, we get $ar\inv{\overline a} \in \bigcap_{P\in\mathfrak P}\N_{\GL(2,q^2)}(P) = \mathrm Z(\GL(2,q^2))$ and there exists $\lambda\in\mathbb F_{q^2}^\times$ such that $r=\lambda \inv a\overline a$. Since $r\in\SL(2,q^2)$, we get $1=\det(r)=\lambda^2\inv d\overline d$. Since $a$ normalizes $T$ and $\gamma_a\rho_r$ stabilizes $\pi_{\mathsf{sq}}$, we get that $\gamma_a\rho_r$ stabilizes $A$ and hence $\Omega$. Compute
		\[ \zv01 \cdot g^ar \cdot \sv 10 = \zv 01 \cdot \mat {\inv d}001 g \lambda \mat{\overline d}001 \cdot \sv 10 = \lambda\overline d \cdot \zv 01 g \sv 10 \text{.} \]
		Hence, $\lambda \overline d\in \mathbb F_q^\times$. But then $d\overline d = (d\inv{\overline d})\overline d^2 = \lambda^2 \overline d^2 = (\lambda\overline d)^2$ and we get $d\in\mathbb F_{q^2}^{\times,\square}$, according to Lemma \ref{fqqq}.
		
		\item Let $d\coloneqq\mat{\det a}001$ and $a'\coloneqq \inv d a\in\SL(2,q^2)$. Set $s\coloneqq \inv{a'}\overline{a'}$. Then $\gamma_{a'}\rho_s=\vartheta_{a'}\in\Gamma$. Assume $\det a\in\mathbb F_{q^2}^{\times,\square}$. According to (c), there exists $x\in\SL(2,q^2)$ such that $\gamma_d\rho_x\in\Gamma$. Set $r\coloneqq x^{a'}s\in\SL(2,q^2)$. Then
		\[\gamma_a\rho_r = \gamma_d\gamma_{a'} \rho_{x^{a'}}\rho_s = \gamma_d\gamma_{a'}\inv{(\gamma_{a'})}\rho_x \gamma_{a'}\rho_s = \gamma_d\rho_x \cdot \gamma_{a'}\rho_s \in \Gamma \text{.} \]
		Analogously, if there exists $r\in\SL(2,q^2)$ such that $\gamma_a\rho_r\in\Gamma$, then we have $\gamma_d\rho_{(r\inv s)^{\inv{a'}}} \in \Gamma$ and the statement follows with (c).
		
		\item If $\gamma_a\rho_r\in\Gamma$ and $\gamma_a\rho_s\in\Gamma$ then $\rho_{(r\inv s)^{\inv a}} = \gamma_a\rho_r(\gamma_a\rho_s)^{-1} \in \Gamma$. According to (b), we thus have $ar\inv s\inv a\in\{\pm\mathds{1}\}$ and hence $s\in\{\pm r\}$. \qedhere
	\end{enumerate}
\end{proof}

\begin{thm} Let	$q\coloneqq p^e$ and $\Gamma\coloneqq \mathrm{Stab}_{\mathfrak A\ltimes R}(\pi_{\mathsf{sq}})$. Then
	\begin{enumerate}[\em (a)]
		\item $\Gamma = (\langle \vartheta_h \mid h\in\SL(2,q^2) \rangle \rtimes \Aut(\mathbb F_{q^2})) \times \langle \rho_{-\mathds{1}} \rangle \cong \mathrm{P\Sigma L}(2,q^2) \times \langle \rho_{-\mathds{1}} \rangle$.
		\item $\#\Gamma = 2e(q^2-1)q^2(q^2+1)$.
	\end{enumerate}
\end{thm}

\begin{proof}
	Directly from Lemma \ref{lemstabqpa}.
\end{proof}

\begin{rem} \leavevmode
	\begin{enumerate}[\em (a)]
		\item Let $\pi$ be a parallelism on $\mathfrak S$ and $\inv{\pi}$ the parallelism on $\mathfrak S$ obtained by inversion. The stabilizers $\mathrm{Stab}_{\mathfrak A\ltimes R}(\pi)$ and $\mathrm{Stab}_{\mathfrak A\ltimes R}(\inv\pi)$ are isomorphic via $\psi \colon \alpha\rho_r\mapsto \alpha\gamma_r\rho_{\inv r}$.
		\item We have $\Gamma\cdot\psi = \Gamma$ and thus $\mathrm{Stab}_{\mathfrak A\ltimes R}(\inv{\pi_{\mathsf{sq}}})=\mathrm{Stab}_{\mathfrak A\ltimes R}(\pi_{\mathsf{sq}})$.
	\end{enumerate}
\end{rem}

\begin{rem}
	For the classical affine $\SL(2,q^2)$-unital $\mathbb U_{C,\mathcal H}$, the stabilizer $\Aut(\mathbb U_{C,\mathcal H}^{\pi_{\mathsf{sq}}})_{[\infty]}=\Aut(\mathbb U_{C,\mathcal H}^{\inv{\pi_{\mathsf{sq}}}})_{[\infty]}$ is isomorphic to $(\mathrm{P\Sigma L}(2,q^2)_C)\times \langle \rho_{-\mathds{1}}\rangle$. The order of $\Aut(\mathbb U_{C,\mathcal H}^{\pi_{\mathsf{sq}}})_{[\infty]}$ equals $4e(q^2+1)$.	
	For each affine $\SL(2,q^2)$-unital $\mathbb U_{S,\mathcal D}$, the order of $\Aut(\mathbb U_{S,\mathcal D}^{\pi_{\mathsf{sq}}})_{[\infty]}=\Aut(\mathbb U_{S,\mathcal D}^{\inv{\pi_{\mathsf{sq}}}})_{[\infty]}$ equals the order of $\Aut(\mathbb U_{S,\mathcal D})_\mathds{1}$.
\end{rem}

\subsection{The Leonids Unitals of Order Four}\label{leonids}

For small orders, parallelisms on $\mathfrak S$ can be found using a computer. For $q\in\{3,5\}$, the classical affine $\SL(2,q)$-unital $\mathbb U_{C,\mathcal H}$ represents the only isomorphism type of affine $\SL(2,q)$-unitals (see \cite[Theorem 3.3]{slu} and \cite[Theorem 6.1]{diss}). An exhaustive search using GAP \cite{gap} shows that up to equivalence, the three known parallelisms $\flat$, $\natural$ and $\pi_{\mathsf{odd}}$ are the only ones existing for $q\in\{3,5\}$ (see \cite[Section 6.2.1]{diss} for details).\bigskip

For $q=4$, there exist two isomorphism types of affine $\SL(2,q)$-unitals, represented by the classical affine $\SL(2,4)$-unital $\mathbb U_{C,\mathcal H}$ and the non-classical affine $\SL(2,4)$-unital described in \cite{slu}, which we denote by $\mathbb U_{C,\mathcal E}$. The full automorphism group of $\mathbb U_{C,\mathcal H}$ is $\mathfrak A_C\ltimes R\cong (C_5\rtimes C_4)\ltimes \SL(2,4)$ and the full automorphism group of $\mathbb U_{C,\mathcal E}$ is of type $C_4\ltimes \SL(2,4)$ and a subgroup of $\Aut(\mathbb U_{C,\mathcal H})$.\bigskip

For order $4$, an exhaustive computer search yields $182$ parallelisms on $\mathfrak S$. We consider the actions of the automorphism groups $\Aut(\mathbb U_{C,\mathcal H})$ and $\Aut(\mathbb U_{C,\mathcal E})$ on the set of $182$ parallelisms found by GAP. The orbit lengths of these actions are listed in Table \ref{paraorbits}.

\begin{table}
	\centering
	\captionabove{Orbit lengths on the set of $182$ parallelisms for order $4$}
	\label{paraorbits}
	\begin{tabular}{cc}\toprule
		$\Aut(\mathbb U_{C,\mathcal H})$	& $\Aut(\mathbb U_{C,\mathcal E})$ \\\midrule
		$1$				& $1$	\\
		$1$				& $1$	\\\midrule
		$30$			& $24$	\\
						& $6$	\\
		$25$			& $20$	\\
						& $5$	\\
		$5$				& $5$	\\
		$60$			& $60$	\\
		$60$			& $60$	\\\bottomrule
	\end{tabular}
\end{table}

Apart from $\flat$ and $\natural$, which are invariant under the action of $\mathfrak A\ltimes R$, there are hence seven pairwise non-equivalent parallelisms in $\mathbb U_{C,\mathcal E}$. We name representatives of those parallelisms $\pi_1,\ldots,\pi_7$ such that $\pi_1$ and $\pi_2$ as well as $\pi_3$ and $\pi_4$ are equivalent in $\mathbb U_{C,\mathcal H}$. Following the numbering in \cite[Table 6.2]{diss}, the parallelism $\pi_7$ is equivalent to $\pi_{\mathsf{sq}}$ and $\pi_6$ is equivalent to $\inv\pi_{\mathsf{sq}}$. In fact, the construction of the parallelism $\pi_{\mathsf{sq}}$ was inspired by the discovery of those parallelisms and partly answers one of the open problems stated in the author's Ph.\,D. thesis (\cite[Question 6 in Chapter 7]{diss}). Completing the two affine unitals with the seven parallelisms $\pi_1,\ldots,\pi_7$, we obtain twelve pairwise non-isomorphic $\SL(2,q)$-unitals, namely

\[ \mathbb U_{C,\mathcal H}^{\pi_2}\text{, } \mathbb U_{C,\mathcal H}^{\pi_4}\text{, } \mathbb U_{C,\mathcal H}^{\pi_5}\text{, } \mathbb U_{C,\mathcal H}^{\pi_6}\text{, } \mathbb U_{C,\mathcal H}^{\pi_7}\text{, } \mathbb U_{C,\mathcal E}^{\pi_1}\text{, } \mathbb U_{C,\mathcal E}^{\pi_2}\text{, } \mathbb U_{C,\mathcal E}^{\pi_3}\text{, } \mathbb U_{C,\mathcal E}^{\pi_4}\text{, }  \mathbb U_{C,\mathcal E}^{\pi_5}\text{, } \mathbb U_{C,\mathcal E}^{\pi_6}\text{, } \mathbb U_{C,\mathcal E}^{\pi_7}\text{.} \]

Since those parallelisms for order $4$ were found during the Leonid meteor shower in November 2018, we call the twelve resulting unitals \textbf{Leonids unitals}.\bigskip

Knowing all affine $\SL(2,4)$-unitals and all parallelisms on $\mathfrak S$ for order~$4$, we state the following

\goodbreak

\begin{thm}[by exhaustive computer search]
	There are exactly $16$ isomorphism types of $\SL(2,4)$-unitals, represented by $\mathbb U_{C,\mathcal H}^\natural$, $\mathbb U_{C,\mathcal H}^\flat$, $\mathbb U_{C,\mathcal E}^\natural$, $\mathbb U_{C,\mathcal E}^\flat$ and the twelve Leonids unitals. \qed
\end{thm}

Recall that the stabilizer of the block at infinity in any $\SL(2,q)$-unital $\mathbb U_{S,\mathcal D}^\pi$ equals the group of those automorphisms of the affine unital $\mathbb U_{S,\mathcal D}$ which stabilize the parallelism~$\pi$. We compute with GAP that in each Leonids unital there are indeed no automorphisms moving the block at infinity, and we are thus able to compute their full automorphism groups as subgroups of $\Aut(\mathbb U_{C,\mathcal H})$ or $\Aut(\mathbb U_{C,\mathcal E})$, respectively. See Table \ref{paraautos} for the isomorphism types of the full automorphism groups of the Leonids unitals. 

\begin{table}
	\centering
	\captionabove{Isomorphism types of the full automorphism groups of the Leonids unitals}
	\label{paraautos}
	\begin{tabular}{ccc}\toprule
		$\mathbb U$				& {$\Aut(\mathbb U)$} & $\#\Aut(\mathbb U)$ \\\midrule
		$\mathbb U_{C,\mathcal H}^{\pi_2}$& $C_4\ltimes D_5$				& $40$ \\
		$\mathbb U_{C,\mathcal H}^{\pi_4}$& $C_4\ltimes A_4$ 			& $48$ \\
		$\mathbb U_{C,\mathcal H}^{\pi_5}$& $C_4\ltimes (A_4\times C_5)$	& $240$	\\
		$\mathbb U_{C,\mathcal H}^{\pi_6}$& $C_5\rtimes C_4$				& $20$	\\
		$\mathbb U_{C,\mathcal H}^{\pi_7}$& $C_5\rtimes C_4$		& $20$	\\\midrule
		$\mathbb U_{C,\mathcal E}^{\pi_1}$& $D_5$					& $10$	\\
		$\mathbb U_{C,\mathcal E}^{\pi_2}$& $ C_4\ltimes D_5$		& $40$ \\
		$\mathbb U_{C,\mathcal E}^{\pi_3}$& $A_4$ 					& $12$	\\
		$\mathbb U_{C,\mathcal E}^{\pi_4}$& $C_4\ltimes A_4$ 		& $48$ \\
		$\mathbb U_{C,\mathcal E}^{\pi_5}$& $C_4\ltimes A_4$ 		& $48$ \\
		$\mathbb U_{C,\mathcal E}^{\pi_6}$& $C_4$ 					& $4$	\\
		$\mathbb U_{C,\mathcal E}^{\pi_7}$& $C_4$ 					& $4$	\\\bottomrule
	\end{tabular}
\end{table}

\begin{rem}
	Regarding the orders of the full automorphism groups of the Leonids unitals in \emph{Table~\ref{paraautos}}, we see that the order of $\Aut(\mathbb U_{C,\mathcal H}^{\pi_5})$ is notably greater than the other orders. Indeed, although $R$ is not contained in its full automorphism group, the unital $\mathbb U_{C,\mathcal H}^{\pi_5}$ admits a group of automorphisms (of isomorphism type $A_4\times C_5$) which acts regularly on the affine points.
\end{rem}\bigskip

\section{Translations}

We consider a special kind of automorphisms of unitals, namely translations. Since in any closure of an affine $\SL(2,q)$-unital, the respective parallelism determines the possible translations with centers on the block at infinity, the study of translations is closely related to the study of parallelisms.

\begin{defi}\label{deftrans}
	A \textbf{translation} with \textbf{center} $c$ of a unital $\mathbb U$ is an automorphism of~$\mathbb U$ that fixes the point $c$ and each block through $c$. The group of all translations with center~$c$ will be denoted by $G_{[c]}$.
	
	We call $c$ a \textbf{translation center} if $G_{[c]}$ acts transitively on the set of points different from~$c$ on any block through $c$.
\end{defi}

\begin{rem}\label{semireg}
	Let $\mathbb U$ be a unital of order $q$ and $c$ a point of $\mathbb U$. Then the group $G_{[c]}$ of all translations of $\mathbb U$ with center $c$ acts semiregularly on the points of $\mathbb U$ different from $c$ and hence semiregularly on the blocks of $\mathbb U$ that are not incident with $c$ \emph{(}see \emph{\cite[Theorem 1.3]{alltrans}}\emph{)}. Thus, $\#G_{[c]}\leq q$ and $c$ is a translation center exactly if $\#G_{[c]}=q$. Further, if a block $B$ of $\mathbb U$ is fixed by $\Aut(\mathbb U)$, then the center of any translation of $\mathbb U$ lies on $B$.
\end{rem}

In the classical unital, each point is a translation center. In any non-classical $\SL(2,q)$-$\natural$-unital and in any $\SL(2,q)$-$\flat$-unital of order $q\geq 3$, the block $[\infty]$ is fixed by the full automorphism group (see \cite[Proposition 3.11 and Theorem 3.16]{autos}) and hence the center of every translation lies on $[\infty]$.\bigskip

We label the points at infinity of any $\SL(2,q)$-$\pi$-unital with the Sylow $p$-subgroups, such that each (affine short) block $P\in\mathfrak P$ through $\mathds{1}$ is incident with the point $P\in[\infty]$. We first compute translations of $\SL(2,q)$-$\natural$-unitals.

\begin{lem}\label{sylow_trans}
	Let $\mathbb U_{S,\mathcal D}^\natural$ be an $\SL(2,q)$-$\natural$-unital. For each $P\in\mathfrak P$, the point $P\in[\infty]$ is a translation center with $G_{[P]}=R_P\coloneqq\{\rho_x\mid x\in P\}$.
\end{lem}

\begin{myproof} 
	For each $P\in\mathfrak P$, the set of blocks through the point $P$ in $\mathbb U_{S,\mathcal D}^\natural$ equals $\{[\infty]\}\cup\{gP\mid g\in\SL(2,q)\}$. Hence, $R_P$ obviously is a group of translations of $\mathbb U_{S,\mathcal D}^\natural$ with center $P$. Since $\# R_P=q$, the statement follows (recall Remark \ref{semireg}).
\end{myproof}

If $\mathbb U_{S,\mathcal D}^\natural$ is not classical, then the translations given in Lemma \ref{sylow_trans} are all translations of~$\mathbb U_{S,\mathcal D}^\natural$.
We need some preparation to compute all possible translations of $\SL(2,q)$-$\pi$-unitals with center on $[\infty]$ and $\pi\in\{\flat,\pi_{\mathsf{odd}},\pi_{\mathsf{sq}},\inv{\pi_{\mathsf{sq}}}\}$ (Theorem \ref{translationen}).

\begin{lem} \label{transalpharhot}
	Let $P\in\mathfrak P$ and let $\tau\in\mathfrak A\ltimes R$ such that $P\cdot\tau = P$. Then $\tau=\alpha\rho_x$ with $\alpha\in\mathfrak A_P$ and $x\in P$.
\end{lem}

\begin{myproof}
	Let $\tau=\alpha\cdot\rho_h\in\mathfrak A\ltimes R$. Then
	$P = P\cdot\tau = (P\cdot\alpha)h$
	and thus $\alpha$ stabilizes $P$ and we have $h\in P$.
\end{myproof}

Let in the following $T\coloneqq \{\mat 1x01 \mid x\in\mathbb F_q \}$ be a fixed Sylow $p$-subgroup of $\SL(2,q)$.

\begin{lem} \label{transt}
	Let $M\coloneqq \{ Tg \mid g\in\N(T) \}$ and $f\coloneqq \mat 10c1$ with $c\in\mathbb F_q^\times$. 
	\begin{enumerate}[\em (a)]
		\item Let $\tau\in \mathfrak A\ltimes R$ such that $X\cdot\tau = X$ for all $X\in M$. Then $\tau =\gamma_a\rho_t$ for some $a\in\N_{\GL(2,q)}(T)$ and $t\in T$.
		\item Let $\tau\in \mathfrak A\ltimes R$ such that $X\cdot\tau = X$ for all $X\in M\cup \{Tf\}$. Then $\tau =\gamma_{\inv t}\rho_t$ for some $t\in T$.
	\end{enumerate}	
\end{lem}

\begin{myproof}
	\begin{enumerate}[(a)]
		\item Since $T\in M$, we have $\tau =\alpha\rho_t$ with $\alpha\in\mathfrak A_T\cong \mathrm P(\N_{\GL(2,q)}(T))\rtimes \Aut(\mathbb F_q)$ and $t\in T$, according to Lemma \ref{transalpharhot}. Note first that right multiplication with any element $t\in T$ fixes each coset $Tg=gT \in M$. If $g=\mat b*0{\inv b} \in \N(T)$, the coset $Tg$ equals $\{ \mat bx0{\inv b}\mid x\in\mathbb F_q\}$. Let $[a]\in \mathrm P(\N_{\GL(2,q)}(T))$. Then we may choose $a=\mat 1y0d$ and for each $Tg$ with $g\in\N(T)$, we have $(Tg)\cdot \gamma_a = (Tg)^a = Tg$. Applying a power~$\varphi^l$ of the Frobenius automorphism $\varphi$ stabilizes the block $Tg = \{\mat bx0{\inv b}\mid x\in \mathbb F_q\}$ exactly if $b^{(p^l)}=b$. Thus, if $\tau=\alpha\rho_t$ stabilizes each block in $M$, then $\alpha=\gamma_a$ with $a\in\N_{\GL(2,q)}(T)$.
		
		\item We already know $\tau=\gamma_a\rho_t$ for some $a\in\N_{\GL(2,q)}(T)$ and $t\in T$. We need to show $\gamma_a=\gamma_{\inv t}$. Since $a$ normalizes $T$, we have $(Tf)\cdot\gamma_a\rho_t = (Tf)^at = Tf^at$. Hence, $(Tf)\cdot\tau = Tf$ exactly if $f^at\inv f \in T$. Let $t\coloneqq \mat 1x01$ and choose without restriction $a=\mat 1y0d$. Then
		\[ \zv 01 \cdot f^at\inv f = \zv{c (\frac{1-c(x+y)}{d} -1)}{1+\frac{c(x+y)}{d}}\text{.}\]
		Since $f^at\inv f$ is in $T$ and $c\neq 0$, we get $y=-x$ and $d=1$ and hence $a=\inv t$.\qedhere
	\end{enumerate}
\end{myproof}

\goodbreak

If $\pi$ is a parallelism on $\mathfrak S$ and $A\in\pi$ a parallel class, we denote
\[\mathcal T_{A}^\pi \coloneqq \{\tau\in\mathfrak A\ltimes R \mid \forall X\in A:X\cdot\tau=X \} \text{.} \]
If $B\in\mathfrak S$ is a short block, we denote by $[B]$ the parallel class of $\pi$ containing $B$.

\goodbreak

\begin{lem} \label{transp} For any $P\in\mathfrak P$, we have: 
	\begin{enumerate}[\em (a)]
		\item $\mathcal T_{[P]}^\flat = \{\gamma_{\inv x}\rho_x \mid x\in P\}$.
		\item $\mathcal T_{[P]}^\pi$ is trivial for $\pi\in\{\pi_{\mathsf{odd}},\pi_{\mathsf{sq}},\inv{\pi_{\mathsf{sq}}}\}$. \label{pi}
	\end{enumerate}
\end{lem}

\begin{myproof}
	Let $\pi\in\{\flat,\pi_{\mathsf{odd}},\pi_{\mathsf{sq}},\inv{\pi_{\mathsf{sq}}}\}$. Following the notation in Lemma \ref{transt}, the parallel class $[T]\in\pi$ contains $M$ and at least one coset $Tf$ with $f=\mat 10c1$ and $c\in\mathbb F_q^\times$. According to Lemma \ref{transt}, we thus get $\mathcal T_{[T]}^\pi \subseteq \{\gamma_{\inv t}\rho_t \mid t\in T\}$. Note that $\gamma_{\inv t}\rho_t$ equals left multiplication with $t$ and that $tgT=gT$ exactly if $g\in\N(T)$ or $t=\mathds{1}$. The parallel class $[T]\in\flat$ equals the set of right cosets of $T$, while in $\pi_{\mathsf{odd}}$, $\pi_{\mathsf{sq}}$, and $\inv{\pi_{\mathsf{sq}}}$, the parallel class $[T]$ contains at least one left coset $gT$, respectively, with $g\notin\N(T)$. Hence, we get $\mathcal T_{[T]}^\flat = \{\gamma_{\inv t}\rho_t \mid t\in T\}$ and $\mathcal T_{[T]}^\pi$ is trivial if $\pi\in\{\pi_{\mathsf{odd}},\pi_{\mathsf{sq}},\inv{\pi_{\mathsf{sq}}}\}$.
	
	For any $\pi\in\{\flat,\pi_{\mathsf{odd}},\pi_{\mathsf{sq}},\inv{\pi_{\mathsf{sq}}}\}$, the stabilizer of $\pi$ in $\mathfrak A\ltimes R$ acts transitively on the parallel classes of $\pi$. Hence, statement (\ref{pi}) follows. If $\pi=\flat$, then each parallel class $[T^h]$ equals $[T]^h$ and we have $(Tg)^h\cdot\tau = (Tg)^h$ exactly if $\tau=\gamma_{t^{-h}}\rho_{t^h}$ for some $t\in T$.
\end{myproof}

\begin{thm}Let $\mathbb U\coloneqq\mathbb U_{S,\mathcal D}^\pi$ be an $\SL(2,q)$-$\pi$-unital. \label{translationen}
	\begin{enumerate}[\em (a)]
		\item If $\pi=\flat$ and $q\geq 3$, then:
			\begin{enumerate}[\em (i)]
				\item If $p=2$, then every non-trivial translation of $\mathbb U$ is given by left multiplication with an involution contained in $\N(S)$. For each Sylow $2$-subgroup $P\in\mathfrak P$, the normalizer $\N(S)$ contains exactly one non-trivial element of $P$.
				\item If $q$ is odd, then $\mathbb U$ does not admit any non-trivial translation.
			\end{enumerate}
		\item If $\pi\in\{\pi_{\mathsf{odd}},\pi_{\mathsf{sq}},\inv{\pi_{\mathsf{sq}}}\}$, then $\mathbb U$ admits no non-trivial translation with center on $[\infty]$. \label{pitr}
	\end{enumerate}
\end{thm}
	
\begin{myproof}
	The translations of $\mathbb U$ with center on the block $[\infty]$ are automorphisms of $\mathbb U$ stabilizing $[\infty]$ and are hence contained in $\mathfrak A_S\ltimes R$. Thus, statement (\ref{pitr}) follows directly from Lemma \ref{transp}.
	
	Now let $\pi=\flat$ and $q\geq 3$. Then the block $[\infty]$ is fixed by the full automorphism group of $\mathbb U$ (see \cite[Theorem 3.16]{autos}) and hence the center of every translation of $\mathbb U$ lies on $[\infty]$ (recall Remark \ref{semireg}). Fix $P\in\mathfrak P$. According to Lemma \ref{transp}, we have
	\[\Gamma_{[P]} \leq \mathcal T_{[P]}^\flat \cap (\mathfrak A_S\ltimes R) = \{\gamma_{\inv x}\rho_x\mid x\in P \text{ such that } \gamma_x \in \mathfrak A_S\} \text{.}\]
	Thus, every possible translation of $\mathbb U$ with center $P$ is given by left multiplication with  $x\in P\cap \N(S)$. The possible isomorphism types of $S$ and its stabilizer in $\Aut(\SL(2,q))$ are given in \cite[Proposition 2.5 and Theorem 2.11]{diss}. According to those, we get that $x=\mathds{1}$ or $\ord(x)=p$ divides $2(q+1)$ and hence $p=2$.
	
	If $p=2$, then $S$ is cyclic (see e.\,g. \cite[Proposition 2.2]{threeaff} or \cite[Proposition 2.5]{diss}) and the normalizer $\N(S)$ is a dihedral group of order $2(q+1)$ containing $S$ as normal subgroup of order $q+1$. There are thus $q+1$ involutions in $\N(S)$, contained in one coset of $S$. For any $P\in\mathfrak P$, the intersection of $S$ and $P$ is trivial and hence  no two non-trivial elements of $P$ are contained in $\N(S)$.
\end{myproof}

\begin{cor}
	For $q$ even, the
	unital $\mathbb U_{C,\mathcal H}^\flat$ admits exactly $q+1$ non-trivial translations, each of order $2$. These translations generate a dihedral group of order $2(q+1)$.\qed
\end{cor}\bigskip


For the Leonids unitals, all translations can explicitly be computed (e.\,g. with GAP).

\begin{prop}[\cite{diss}, Proposition 6.17] All translations of the Leonids unitals are known:
	\begin{enumerate}[\em (i)]
		\item $\mathbb U_{C,\mathcal H}^{\pi_2}$ and $\mathbb U_{C,\mathcal E}^{\pi_2}$ admit exactly one non-trivial translation each, of order~$2$.
		\item $\mathbb U_{C,\mathcal H}^{\pi_4}$, $\mathbb U_{C,\mathcal H}^{\pi_5}$, $\mathbb U_{C,\mathcal E}^{\pi_3}$, $\mathbb U_{C,\mathcal E}^{\pi_4}$ and $\mathbb U_{C,\mathcal E}^{\pi_5}$ admit exactly three non-trivial translations each. These translations, respectively, have a common center (which is thus a translation center) and generate an elementary abelian group of order~$4$.\qed
	\end{enumerate}
\end{prop}\bigskip\bigskip

\textbf{Acknowledgment.} The author wishes to warmly thank her thesis advisor Markus J. Stroppel for his highly valuable support in each phase of this research.

	\printbibliography
\end{document}